\documentclass[11pt,reqno]{amsart}
\usepackage[english]{babel}
\usepackage[a4paper,twoside,marginparwidth=50pt]{geometry}
\usepackage{microtype}
\usepackage{enumerate}
\usepackage{csquotes}
\usepackage{longtable}
\usepackage{bbm}
\usepackage{nicefrac}
\usepackage{array}
\usepackage{marginnote}

\usepackage{xspace}
\patchcmd{\subsection}{-.5em}{.5em}{}{}
\usepackage[hyperfootnotes=false]{hyperref} 
\usepackage{color}
\hypersetup{
	colorlinks = true,
	linkcolor = {blue},
	urlcolor = {red},
	citecolor = {blue}
}

\makeatletter
\newcommand{\mylabel}[2]{#2\def\@currentlabel{#2}\label{#1}}
\makeatother

\makeatletter
\renewcommand{\tocsection}[3]{
  \indentlabel{\@ifnotempty{#2}{\ignorespaces#1 #2\quad}}\bfseries#3}
\renewcommand{\tocsubsection}[3]{
  \indentlabel{\@ifnotempty{#2}{\ignorespaces#1 #2\quad}}#3}

\newcommand\@dotsep{4.5}
\def\@tocline#1#2#3#4#5#6#7{\relax
  \ifnum #1>\c@tocdepth
  \else
    \par \addpenalty\@secpenalty\addvspace{#2}
    \begingroup \hyphenpenalty\@M
    \@ifempty{#4}{
      \@tempdima\csname r@tocindent\number#1\endcsname\relax
    }{
      \@tempdima#4\relax
    }
    \parindent\z@ \leftskip#3\relax \advance\leftskip\@tempdima\relax
    \rightskip\@pnumwidth plus1em \parfillskip-\@pnumwidth
    #5\leavevmode\hskip-\@tempdima{#6}\nobreak
    \leaders\hbox{$\m@th\mkern \@dotsep mu\hbox{.}\mkern \@dotsep mu$}\hfill
    \nobreak
    \hbox to\@pnumwidth{\@tocpagenum{\ifnum#1=1\bfseries\fi#7}}\par
    \nobreak
    \endgroup
  \fi}
\AtBeginDocument{
\expandafter\renewcommand\csname r@tocindent0\endcsname{0pt}
}
\def\l@subsection{\@tocline{2}{0pt}{2.5pc}{5pc}{}}
\makeatother

\usepackage{amsmath}
\usepackage{amsthm}
\usepackage{amssymb}
\usepackage{mathrsfs} 
\usepackage{eucal}
\usepackage{mathtools}

\newcounter{results}[section]

\theoremstyle{plain}
\newtheorem{theorem}[results]{Theorem}
\newtheorem{lemma}[results]{Lemma}
\newtheorem{proposition}[results]{Proposition}

\theoremstyle{remark}
\newtheorem{remark}[results]{Remark}
\newtheorem{example}[results]{Example}

\theoremstyle{definition}
\newtheorem{definition}[results]{Definition}

\theoremstyle{remark}
\newtheorem{pty}{Property}
\renewcommand{\thepty}{P\arabic{pty}}

\numberwithin{equation}{section}

\newcommand{\norm}[1]{ \left \| #1 \right \|}

\newcommand{\N}{\mathbb{N}}
\newcommand{\R}{\mathbb{R}}
\renewcommand{\AA}{\mathscr{A}}
\newcommand{\EE}{\mathscr{E}}
\newcommand{\PP}{\mathbb{P}}
\renewcommand{\SS}{\mathscr{S}}

\newcommand{\Ggamma}{{\mbox{\boldmath$\Gamma$}}}

\newcommand{\sfd}{{\sf d}}

\newcommand{\rmC}{{\mathrm C}}

\newcommand{\argmin}{\mathop{\rm argmin}\limits}

\newcommand{\Lip}{\mathop{\rm Lip}\nolimits}          
\newcommand{\Lipb}{\mathop{\rm Lip}_b\nolimits}          

\newcommand{\lip}{\mathop{\rm lip}\nolimits}          
\newcommand{\eps}{\varepsilon}  

\newcommand{\JJ}{{\mathcal J}}

\newcommand{\HK}{\mathsf H\!\!\mathsf K}

\newcommand{\prob}{\mathcal P}

\newcommand{\mm}{\mathfrak m}
\newcommand{\pp}{\mathfrak p}

\newcommand{\CE}{\mathsf{C\kern-1pt E}}

\newcommand{\pCE}{\mathsf{pC\kern-1pt E}}

\newcommand{\de}{\, \mathrm d}

\title[Approximation in Metric Sobolev Spaces]{Approximation in Metric Sobolev Spaces: A General Framework}

\author{Massimo Fornasier}
\address{Massimo Fornasier: TUM School of Computation, Information, and Technology - Department of Mathematics, Boltzmannstrasse 3, 85748 Garching bei M\"unchen (Germany) \newline \& 
TUM Institute for Advanced Studies
\newline \&
Munich Data Science Institute \newline \& Munich Center for Machine Learning}
\email{massimo.fornasier@cit.tum.de}

\author{Giacomo Enrico Sodini}
\address{Giacomo Enrico Sodini:  Institute of Analysis and Scientific Computing, TU Wien, Wiedner Haupt-
strasse 8-10, A-1040 Vienna, Austria}
\email{giacomo.sodini@tuwien.ac.at}


\begin{document} 

\begin{abstract}
In our recent work \cite{ForHeiSod25}, we introduced a numerical framework for approximating Sobolev functions on Wasserstein spaces from finite samples, leveraging structural properties established in \cite{ForSavSod23}. The present paper demonstrates that this methodology extends far beyond that specific setting. We identify a general class of metric measure spaces---including weighted Riemannian manifolds and spaces of measures equipped with the Hellinger--Kantorovich distance---for which the key hypotheses of Hilbertianity and the existence of a computable algebra of Lipschitz functions hold. Within this abstract framework, we recover and generalize the core approximation results of \cite{ForHeiSod25} for recovering functions from random point evaluations. Our main contribution is to show that the combination of theoretical foundations from \cite{ForSavSod23} and algorithmic strategies from \cite{ForHeiSod25} is robust enough to apply to a wide variety of infinite-dimensional spaces of current interest.
\end{abstract}

\maketitle
\tableofcontents
\thispagestyle{empty}

\section{Introduction}

The primary goal of this work is to advance the program of bringing abstract functional analysis and, more specifically, the theory of metric measure spaces into the realm of practical computability. This is motivated by the growing need for robust mathematical foundations for data analysis and machine learning in settings where data is intrinsically infinite-dimensional and may come with their own intrinsic metric (e.g., distances between images or shapes). In our previous work \cite{ForHeiSod25}, we showed the feasibility and power of this approach by developing a comprehensive framework for the numerical approximation of Sobolev-smooth functions defined on the Wasserstein space of probability measures. A central and successful application was the efficient approximation of the Wasserstein distance function itself, a task of significant importance in fields such as image comparison, retrieval, and recognition, where our method offered substantial computational speedups after training.

The success of the constructive methods in \cite{ForHeiSod25} was critically based on the deep structural properties of the underlying metric measure space, which were established in our preceding theoretical work \cite{ForSavSod23}. A key finding was that, for any finite reference measure $\mathfrak{m}$ on the Wasserstein space $(\prob_2(\R^d), \mathsf{W}_2)$, the associated Cheeger energy is a quadratic form, making the Sobolev space $H^{1,2}$ a Hilbert space. At this introductory level, let us just mention that the Cheeger energy serves as a surrogate of the integral norm of the gradient in the context of metric spaces, see \eqref{eq:CheegerEnergy} for the precise definition.

This Hilbertianity, combined with the explicit computability of the pre-Cheeger energy (see \eqref{eq:pce}) on a sufficiently large and practically implementable algebra of Lipschitz functions (cylinder functions), provided the essential tool for the approximation algorithms in \cite{ForHeiSod25}. A crucial step in those algorithms involved going from a continuous (or diffused) probability measure $\mathfrak{p}$ on the Wasserstein space to its discrete (or empirical)  approximations $\mathfrak{p}_N$, drawn from random samples.

This type of ``discrete-to-continuous" transition is a well-established theme in approximation theory and numerical analysis, often explored to develop methods that are both theoretically sound and computationally viable. Notably, the work of \cite{AmbColDma} shares a similar spirit, investigating approximation problems in metric measure spaces using a discretization of the underlying space. However, a cornerstone of their analysis is the metric doubling property of the underlying space. 
Under such condition one can define partitions of the space with finitely many neighboring sets, whose number is uniform with respect to the resolution of the partition. Then a discrete Cheeger energy can be defined in terms of finite differences of integral evaluations over neighboring sets. As the resolution of the partition increases,  \cite[Theorem 7.4]{AmbColDma} proves that such discrete Cheeger energy tends -- in the sense of $\Gamma$-convergence --  to become equivalent, up to multiplicative constants, to the Cheeger energy.
While the doubling condition is powerful, it also means that the metric space is intrinsically finite-dimensional, corresponding to the finitely many neighboring sets of partitions.  Therefore the approximation proposed in \cite{AmbColDma}  can neither be applied in a genuinely infinite-dimensional metric space, such as the  Wasserstein space, that was our focus in \cite{ForHeiSod25,ForSavSod23}, nor in many other infinite-dimensional metric spaces of  interest in modern application areas such as machine learning.

This observation naturally leads to the central question of the present paper: can the methodology pioneered in \cite{ForHeiSod25} be abstracted and generalized beyond the specific case of the Wasserstein space to a broader class of metric measure spaces? The answer, as we will show, is affirmative.

We begin by formally introducing a class of metric measure spaces for which the key hypotheses enabling our earlier work hold. These hypotheses, which we clarify and discuss in detail, include the Hilbertianity of the pre-Cheeger energy and the existence of a sufficiently rich algebra of Lipschitz functions for which the pre-Cheeger energy is computable. We then present a series of relevant and non-trivial examples that fit this framework. Beyond the already-treated Wasserstein space $(\prob_2(\mathbb{R}^d), \mathsf{W}_2)$, these include weighted Riemannian manifolds and spaces of measures endowed with the Hellinger--Kantorovich distance. This latter example shows the applicability of our framework to other optimal transport geometries.

Within this general setting, we then show how to recover and extend some of the core approximation results from \cite{ForHeiSod25}. Specifically, we focus on the problem of approximating an unknown function from its values known only at a finite set of random points, distributed according to the underlying reference measure. We tackle this problem in two ways. First, we approach the approximation by working at the level of the pre-Cheeger energy. This is advantageous in concrete applications, as this form is often more amenable to direct computation and numerical implementation, for instance, via neural networks as demonstrated in \cite{ForHeiSod25}. Second, by introducing an additional approximation parameter to handle the squared field representing the Cheeger energy, we demonstrate that one can work directly at the level of the Cheeger energy itself. 

This paper is primarily intended as a review and synthesis, aiming to demonstrate the robustness and generality of the techniques developed in \cite{ForHeiSod25}. While the results are not entirely surprising given the foundational work in \cite{ForSavSod23} and the algorithmic framework in \cite{ForHeiSod25}, the primary contribution lies in showing that this powerful computational-analytical framework can be coherently and meaningfully extended to encompass a much wider class of infinite-dimensional metric measure spaces, including several of significant interest in modern applications.

Finally, this work is situated within the broader, and still partly unexplored, context of analysis in infinite-dimensional spaces, and in particular in spaces of measures. By demonstrating the robustness and applicability of a computational-analytical framework beyond the specific case of the Wasserstein space, we hope to highlight the richness of this field and encourage further investigation. Many fundamental questions, such as the existence of compact Sobolev embeddings, the validity of Poincaré inequalities, and the analysis of intrinsic distances induced by Dirichlet forms and their associated stochastic processes, remain interesting and open questions for future research in this infinite-dimensional setting.

\section{The relaxation approach to Metric Sobolev spaces}
Several approaches to defining Sobolev spaces on metric-measure spaces have been studied in recent years. Here, a metric-measure space is a complete and separable metric space $(\SS, \sfd)$ endowed with a non-negative, finite Borel measure $\mm$ on $(\SS,\sfd)$. This family of spaces includes, among others, weighted Banach spaces, Riemannian and Finsler manifolds, RCD spaces, Carnot groups, and Alexandrov spaces. At least four different approaches exist for defining a Sobolev space on such spaces:
\begin{enumerate}
    \item[(i)] The Newtonian approach, due to Shanmugalingam \cite{Sha00} after Koskela-MacManus \cite{KosMMa98};
    \item[(ii)] The Beppo-Levi approach, due to Ambrosio, Gigli, and Savaré \cite{AmbGigSav13, AmbGigSav14}, later revisited by Savaré in \cite{Sav22}; this approach studies the behavior of functions along suitable families of absolutely continuous curves, and thus relies on the metric derivative of such curves, whose existence goes back to Ambrosio \cite[Theorem 2.2(i)]{Amb90} (see also \cite[Theorem 1.1.2]{AmbGigSav08});
    \item[(iii)] The derivation approach, due to Di Marino \cite{DMa14};
    \item[(iv)]The relaxation approach, originally proposed by Cheeger \cite{Che99}, and later revisited by Ambrosio, Gigli, and Savar\'e in \cite{AmbGigSav13, AmbGigSav14}.
\end{enumerate}
Under suitable assumptions, the above approaches are equivalent and give rise to the same object: the metric Sobolev space $H^{1,2}(\SS, \sfd, \mm)$, a Banach space of $L^2$ functions that are differentiable in a suitable sense. We cite the two monographs \cite{HeiJuhSha15, BjoBjo11} for an overview on the matter and, in particular, refer to \cite{AmbIkoLucPas24, AmbIkoLucPas25} for a comparison of the mentioned approaches. We also note that our exposition is limited to the exponent $p=2$ in the Sobolev space; analogous constructions could be carried out for $p\ne2$, as in the classical theory of Sobolev spaces on Euclidean and Riemannian settings.

Here we focus on the fourth approach, which is perhaps the easiest to present from scratch and has already been employed in the same spirit as the present paper in the work \cite{ForHeiSod25}, which we take inspiration from for most of our constructions.

We briefly introduce a few relevant objects: for a $\sfd$-Lipschitz and bounded function $f \in \Lip_b(\SS, \sfd)$ and for any given $A \subset \SS$, the quantity $\Lip(f, \sfd, A)$ is the $\sfd$-Lipschitz constant of $f$ in $A$ defined as
\[
\Lip(f, \sfd, A) := \sup \left \{ \frac{|f(y)-f(z)|}{\sfd(y,z)} : y,z \in A, \, y\ne z\right \}.
\]
We then define the $\sfd$-\emph{Lipschitz asymptotic constant} of $f$ as
\begin{equation}
    \lip_\sfd f(x):= \lim_{r \downarrow 0} \Lip(f, \sfd, B_r(x)), \quad x \in \SS,
\end{equation}
where $B_r(x):= \{ y \in \SS : \sfd(x,y)<r\}$ is the $\sfd$-open ball of radius $r>0$.
It is clear that $\lip_\sfd f$ serves as a surrogate for the norm of the gradient: in fact, for a smooth function $f \in \rmC^\infty(\R^d)$ it holds
\[
\lip_{\sfd_e} f (x) = |\nabla f (x)|,
\]
where $\sfd_e$ is the distance induced by the Euclidean norm $|\cdot|$ in $\R^d$. Clearly $\lip_\sfd$ is a pointwise and purely metric object: its integrated counterpart is the so-called \emph{pre-Cheeger energy}:
\begin{equation}\label{eq:pce}
    \pCE_{2,\mm} (f) := \int_\SS (\lip_\sfd f)^2 \de\mm, \quad f \in \Lipb(\SS, \sfd).
\end{equation}
Note that $\pCE$ also depends on the underlying distance $\sfd$, but we are not going to stress this dependence, which will always be clear from the context. At this point, $\pCE$ is defined only for $\sfd$-Lipschitz and bounded functions, and there is no reason for it to be lower semicontinuous w.r.t.~the $L^2$ topology. Therefore, we introduce its $L^2$ relaxation, the \emph{Cheeger energy}, given by
\begin{equation}\label{eq:relaxintroche}
  \CE_{2,\mm}(f) = \inf \left \{ \liminf_{n \to \infty}
    \pCE_{2,\mm}(f_n) : (f_n)_n \subset \Lipb(\SS, \sfd), \, f_n \to f \text{ in } L^2(\SS,\mm)
  \right \}.
\end{equation}
The resulting Sobolev space $H^{1,2}(\SS, \sfd, \mm)$ is then the vector space of functions $f\in L^2(\SS, \mm)$ with finite Cheeger energy endowed with the norm
\begin{align} \label{eq:CheegerEnergy} |f|^2_{H^{1,2}(\SS, \sfd, \mm)}:= \int_\SS |f|^2 \de\mm + \CE_{2,\mm}(f),
\end{align}
which makes it a Banach space (cf.~\cite[Theorem 2.1.17]{GigPas20}). Note that, for every $f \in \Lipb(\SS, \sfd)$, we have that $\lip_\sfd f \le \Lip(f,\sfd, \SS)$ so that $\Lipb(\SS, \sfd) \subset H^{1,2}(\SS, \sfd, \mm)$.

\subsection{Infinitesimal Hilbertianity}
In general, $H^{1,2}(\SS, \sfd, \mm)$ is not a Hilbert space. For instance, if $\SS=\R^d$ with $d \ge 2$, $\sfd(x,y):=|x-y|_\infty$, $x,y \in \R^d$, is the distance induced by the infinity norm, and $\mm$ is a Gaussian measure on $\R^d$, then $H^{1,2}(\SS, \sfd, \mm)$ is not a Hilbert space. When the Hilbertian property is satisfied, we say that the metric-measure space $(\SS, \sfd, \mm)$ is \emph{infinitesimally Hilbertian} (this notion was introduced in \cite[Definition 4.19]{Gig15}), a crucial property which is the starting point of the theory of RCD spaces. When a (complete and separable) metric space $(\SS, \sfd)$ is such that $(\SS, \sfd, \mm)$ is infinitesimally Hilbertian for every non-negative, finite, Borel measure $\mm$ on $(\SS, \sfd)$, then we say that $(\SS, \sfd)$ is \emph{universally infinitesimally Hilbertian} (this notion first appeared in \cite{DMaGigPasSou21}). Weighted Riemannian manifolds, weighted Banach spaces, and other non-trivial spaces fall into this category, as we will see in the following.

It has been shown in \cite[Theorem 2.17]{ForSavSod23} that a sufficient condition for a metric-measure space $(\SS, \sfd, \mm)$ to be infinitesimally Hilbertian is given by the existence of a unital subalgebra $\AA \subset \Lipb(\SS, \sfd)$ (that is, $\AA$ is a subalgebra of $\Lip_b(\SS, \sfd)$ and contains constant functions) with the two following properties:
\begin{pty}[Density in energy]\label{it:1}
$\AA$ is dense in energy in the metric Sobolev space: for every $f \in H^{1,2}(\SS, \sfd, \mm)$ there exists a sequence $(f_n)_n \subset \AA$ such that
\[
f_n \to f \text{ in } L^2(\SS, \mm) \quad \text{and} \quad \lim_{n \to \infty} \pCE_{2,\mm}(f_n) = \CE_{2, \mm}(f).
\]
\end{pty}
\begin{pty}[Quadraticity]\label{it:2}
$(\pCE_{2,\mm}, \AA)$ is quadratic, equivalently:
    \[
    \pCE_{2,\mm}(f+h) + \pCE_{2,\mm}(f-h) = 2\pCE_{2,\mm}(f)+2\pCE_{2,\mm}(h) \quad \text{ for every } f,h \in \AA.
    \]
\end{pty}
Note that condition \ref{it:2} implies, by polarization, that $\pCE$ is a quadratic form: more precisely, if we set
\begin{equation}\label{eq:polar}
\pCE_{2,\mm}(f,h) := \frac14 \Big( \pCE_{2,\mm}(f+h) - \pCE_{2,\mm}(f-h) \Big), \quad f,h \in \Lipb(\SS,\sfd),
\end{equation}
then \ref{it:2} implies that $\pCE_{2,\mm}(\cdot,\cdot)$ is a symmetric bilinear form on $\AA$ satisfying $\pCE_{2,\mm}(f,f) = \pCE_{2,\mm}(f)$ for every $f \in \AA$. We remark that \eqref{eq:polar} makes sense for every non-negative and finite Borel measure $\mm$ on $(\SS,\sfd)$; in Section~\ref{sec:regbdfunc} we will also use it with $\mm$ replaced by an empirical measure.
On the other hand, condition \ref{it:1} states that we can approximate the Cheeger energy of every $f \in H^{1,2}(\SS, \sfd, \mm)$ via the pre-Cheeger energy of a sequence of functions in $\AA$, i.e., that the infimum in \eqref{eq:relaxintroche} can be computed by considering only sequences in $\AA$. As a consequence of \ref{it:1} and \ref{it:2}, $\AA$ is also \emph{strongly dense} in the metric Sobolev space, i.e., for every $f \in H^{1,2}(\SS, \sfd, \mm)$ there exists a sequence $(f_n)_n \subset \AA$ such that $f_n \to f$ in $H^{1,2}(\SS, \sfd, \mm)$ (cf.~\cite[Theorem 2.17]{ForSavSod23}).

In general, the Cheeger energy can be difficult to compute or is defined only implicitly. In contrast, the pre-Cheeger energy often takes a more explicit form, for instance, when the following condition (which is stronger than \ref{it:2}) holds:
{\addtocounter{pty}{-1}\renewcommand{\thepty}{P2$'$}\def\theHpty{P2prime}%
\begin{pty}[Pointwise representation]\label{it:2p}
There exists a symmetric bilinear map $\Ggamma: \AA \times \AA \to L^1(\SS, \mm)$ such that
    \[ (\lip_\sfd f(x))^2 = \Ggamma(f,f)(x) \quad \text{ for every } f \in \AA \text{ and $\mm$-a.e.~} x \in \SS.\]
\end{pty}}
\noindent If \ref{it:2p} holds, then, by integration and polarization (recall \eqref{eq:polar}), for every $f,h \in \AA$ we get
\[
\pCE_{2,\mm}(f) = \int_\SS \Ggamma(f,f) \de \mm, \qquad \pCE_{2,\mm}(f,h) = \int_\SS \Ggamma(f,h) \de \mm,
\]
so that \ref{it:2} holds as well.

We now list several examples where conditions \ref{it:1} and \ref{it:2p} (and consequently condition \ref{it:2}) are satisfied. In all the following examples, the metric space under consideration is universally infinitesimally Hilbertian.

\begin{example}[Weighted Euclidean spaces]\label{ex:eu}
    In the Euclidean space $\R^d$, we consider the distance $\sfd_e$ induced by the Euclidean norm, and a non-negative, finite Borel measure $\mm$ on $(\R^d, \sfd_e)$. Then one can take $\AA = \rmC_c^\infty(\R^d)$, the space of smooth and compactly supported functions on $\R^d$, together with the map
    \[
        \Ggamma(f,h) := \langle \nabla f, \nabla h \rangle, \quad f,h \in \rmC_c^\infty(\R^d),
    \]
    where $\nabla$ is the Euclidean gradient. Both properties \ref{it:1} and \ref{it:2p} can be found in e.g.~\cite[Example 5.3.8]{Sav22}. The Hilbertianity of $H^{1,2}(\R^d, \sfd_e, \mm)$ can be found in e.g.~\cite[Theorem 11]{DMaLucPas20}.
\end{example}

\begin{example}[Weighted Riemannian manifolds]\label{ex:ri}
    Let $(M,g)$ be a (smooth, complete) Riemannian manifold with induced Riemannian metric $\sfd_g$, and let $\mm$ be a non-negative, finite Borel measure on $(M, \sfd_g)$. As in the Euclidean case, we can take $\AA = \rmC_c^\infty(M)$, the space of smooth and compactly supported functions on $M$, together with the map
    \[
        \Ggamma(f,h) := g( \nabla^g f, \nabla^g h), \quad f,h \in \rmC^\infty_c(M),
    \]
    where $\nabla^g$ is the gradient induced by $g$. Both properties \ref{it:1} and \ref{it:2p} can be found in e.g.~\cite[Example 5.3.8]{Sav22}. The Hilbertianity of $H^{1,2}(M, \sfd_g, \mm)$ can be found in \cite[Theorem 3.11]{LucPas20}.
\end{example}

\begin{example}[Weighted Hilbert spaces]\label{ex:hi}
    If $(H, \langle \cdot, \cdot \rangle)$ is a separable Hilbert space with induced distance $\sfd_H$, for any non-negative, finite Borel measure $\mm$ on $(H, \sfd_H)$, we can take as $\AA$ the space $\rmC^1_b(H)$ of continuously Fréchet-differentiable functions with bounded Fréchet differential on $H$, with the map
    \[
        \Ggamma(f,h) := \langle Df, Dh \rangle, \quad f,h \in \rmC^1_b(H),
    \]
    where $D$ is the Fréchet differential. Properties \ref{it:1} and \ref{it:2p} are contained in \cite[Example 5.3.9]{Sav22} and the Hilbertianity of $H^{1,2}(H, \sfd_H, \mm)$ in \cite[Corollary 5.3.11]{Sav22}.
\end{example}

Before introducing the next three examples, we define the algebra that will play the role of $\AA$ in all of them: the space of so-called \emph{cylinder functions}. Let $\mathsf{E}$ be a Euclidean space, a Riemannian manifold, or a separable Hilbert space, and let $\mathcal{M}_+(\mathsf E)$ be the space of non-negative, finite Borel measures on $(\mathsf E, \sfd_{\mathsf E})$, where $\sfd_\mathsf{E}$ is the Euclidean, the Riemannian, or the Hilbertian distance, respectively. For a bounded Borel function $f: \mathsf{E} \to \R$, we define the map $f^\star: \mathcal{M}_+(\mathsf E) \to \R$ by
\[
    f^\star(\mu) := \int_{\mathsf{E}} f \de \mu, \quad \mu \in \mathcal{M}_+(\mathsf E).
\]
A function $u: \mathcal{M}_+(\mathsf E) \to \R$ is called a cylinder function if 
\begin{equation}\label{eq:thecyl}
    u = F \circ (f_0^\star, f_1^\star, \dots, f_k^\star),
\end{equation}
for some $F \in \rmC^\infty_c(\R^{k+1})$, $k \in \N_0$, smooth (in the appropriate sense, as in the examples above) functions $f_i$ on $\mathsf{E}$ for $i=1, \dots, k$, and $f_0$ equal to the constant function $1$. The resulting algebra of functions is denoted by $\mathsf{Cyl}(\mathsf E)$. We mention that the following three examples belong to the larger framework of \emph{Unbalanced Optimal Transport}, see \cite{SavSod24}.

\begin{example}[$L^2$-Hellinger spaces]\label{ex:he}
    Let $\mathsf{He}_2$ be the $L^2$-\emph{Hellinger distance} on $\mathcal{M}_+(\mathsf E)$ given by
    \[
        \mathsf{He}_2(\mu, \nu)^2 := \int_{\mathsf E} \left | \sqrt{ \frac{\de \mu}{\de \eta}}- \sqrt{\frac{\de \nu}{\de \eta}} \right |^2 \de \eta, \quad \mu, \nu \in \mathcal{M}_+(\mathsf E),
    \]
    where $\eta \in \mathcal{M}_+(\mathsf E)$ is any measure such that $\mu, \nu \ll \eta$, and $\de \mu / \de \eta$ (resp.~$\de \nu / \de \eta$) denotes the Radon--Nikodym derivative of $\mu$ (resp.~$\nu$) with respect to $\eta$. We stress that, by $2$-homogeneity, the above definition does not depend on the choice of $\eta$, so one can always take e.g.~$\eta = \mu + \nu$. It can be shown that $(\mathcal{M}_+(\mathsf E), \mathsf{He}_2)$ is a complete and separable metric space whose topology is that of total variation; see e.g.~\cite{LuiSav21}. Let $\mm$ be any non-negative, finite Borel measure on $\mathcal{M}_+(\mathsf E)$. For $\AA = \mathsf{Cyl}(\mathsf E)$, the map $\Ggamma$ is given by
    \begin{equation}\label{eq:gammaver}
        \Ggamma^{\textrm{ver}}(u,v)(\mu) := \int_{\mathsf E} (\boldsymbol{\nabla}^{\textrm{ver}} u)_\mu(x) (\boldsymbol{\nabla}^{\textrm{ver}} v)_\mu(x) \de \mu(x), \quad \mu \in \mathcal{M}_+(\mathsf E), \, u,v \in \mathsf{Cyl}(\mathsf E),
    \end{equation}
    where $\boldsymbol{\nabla}^{\textrm{ver}}$ is the \emph{vertical gradient} given, for a cylinder function $u$ as in \eqref{eq:thecyl}, by
    \[
        (\boldsymbol{\nabla}^{\textrm{ver}} u)_\mu(x) := \sum_{i=0}^k \partial_i F ((f_0^\star, f_1^\star, \dots, f_k^\star)(\mu)) f_i(x), \quad x \in \mathsf{E}, \, \mu \in \mathcal{M}_+(\mathsf E).
    \]
    The density property \ref{it:1} and the Hilbertianity of $H^{1,2}(\mathcal{M}_+(\mathsf E), \mathsf{He}_2, \mm)$ are proven in \cite[Corollary 4.20]{DscSod25}, while property \ref{it:2p} is shown in \cite[Corollary 4.8]{DscSod25}.
\end{example}

\begin{example}[$L^2$-Kantorovich--Rubinstein--Wasserstein spaces]\label{ex:w2}
    Let $\prob_2(\mathsf E)$ be the space of Borel probability measures $\mu$ on $\mathsf E$ with finite $2$-moment, i.e.,
    \[
        \int_{\mathsf E} \sfd_{\mathsf E}(x,x_o)^2 \de \mu(x) < \infty \quad \text{ for some (hence for any) } x_o \in \mathsf E.
    \]
    On $\prob_2(\mathsf E)$ we consider the $L^2$-\emph{Kantorovich--Rubinstein--Wasserstein} distance $\mathsf{W}_2$ given by
    \[
        \mathsf{W}_2(\mu, \nu)^2 := \inf \left \{ \int_{\mathsf E \times \mathsf E} \sfd_{\mathsf E}(x,y)^2 \de \pi(x,y) \, : \, \pi \in \Pi(\mu, \nu)\right \}, \quad \mu, \nu \in \prob_2(\mathsf E), 
    \]
    where $\Pi(\mu, \nu)$ is the set of Borel probability measures on $\mathsf E \times \mathsf E$ having $\mu$ and $\nu$ as first and second marginal, respectively. It is well known that $(\prob_2(\mathsf E), \mathsf{W}_2)$ is a complete and separable metric space whose topology is stronger than the weak (or narrow) one; see e.g.~\cite{AmbGigSav08}. Let $\mm$ be any non-negative, finite Borel measure on $\prob_2(\mathsf E)$. Since $\prob_2(\mathsf E) \subset \mathcal{M}_+(\mathsf E)$, we can still consider the algebra $\AA = \mathsf{Cyl}(\mathsf E)$, with the map $\Ggamma$ given by
    \begin{equation}\label{eq:gammahor}
        \Ggamma^{\textrm{hor}}(u,v)(\mu) := \int_{\mathsf E} \langle (\boldsymbol{\nabla}^{\textrm{hor}} u)_\mu(x) ,(\boldsymbol{\nabla}^{\textrm{hor}} v)_\mu(x)\rangle_{\mathsf{T}_x \mathsf{E}} \de \mu(x), \quad \mu \in \prob_2(\mathsf E), \, u,v \in \mathsf{Cyl}(\mathsf E),
    \end{equation}
    where $\boldsymbol{\nabla}^{\textrm{hor}}$ is the \emph{horizontal gradient} given, for a cylinder function $u$ as in \eqref{eq:thecyl}, by
    \[
        (\boldsymbol{\nabla}^{\textrm{hor}} u)_\mu(x) := \sum_{i=0}^k \partial_i F ((f_0^\star, f_1^\star, \dots, f_k^\star)(\mu)) \nabla^{\mathsf E} f_i(x), \quad x \in \mathsf{E}, \, \mu \in \prob_2(\mathsf E),
    \]
    where $\nabla^{\mathsf E}$ is the appropriate differential or gradient operator on $\mathsf E$. The density property \ref{it:1} and the Hilbertianity of $H^{1,2}(\prob_2(\mathsf E), \mathsf{W}_2, \mm)$ are contained in \cite[Theorem 4.10, Corollary 4.11]{ForSavSod23}, while property \ref{it:2p} is \cite[Proposition 4.9]{ForSavSod23}. These results are further extended in \cite{Sod23, DscSod25}.
\end{example}

\begin{example}[Hellinger--Kantorovich space]\label{ex:hk}
    On the space $\mathcal{M}_+(\mathsf E)$ we consider a different distance, namely the so-called \emph{Hellinger--Kantorovich} (also known as Wasserstein--Fisher--Rao) distance given by
    \[
        \HK(\mu, \nu)^2 := \inf \left \{ \int_0^1 \int_{\mathsf{E}} \left [ |v_t|^2 + \frac{1}{4} |w_t|^2 \right ] \de \eta_t(x) \de t : \partial_t \eta + \textrm{div}(v_t \eta_t )= w_t \eta_t,\,\, \eta_0 = \mu,\,\eta_1 = \nu \right \},
    \]
    where the infimum is taken over narrowly continuous curves $\eta: [0,1] \to \mathcal{M}_+(\mathsf E)$ satisfying the PDE above in the distributional sense on $[0,1] \times \mathsf E$. The space $(\mathcal{M}_+(\mathsf E), \HK)$ is a complete and separable metric space; see the works \cite{KonMonVor16, ChiPeySchVia18, LieMieSav18} where it was introduced, and \cite{DPoSodTam25} for the connection between Wasserstein, Hellinger, and Hellinger--Kantorovich distances. Considering again the algebra $\AA$ of cylinder functions, the map $\Ggamma$ in this case is given by
    \[
        \Ggamma(u,v) := \Ggamma^{\textrm{hor}}(u,v) + 4 \Ggamma^{\textrm{ver}}(u,v), \quad u,v \in \mathsf{Cyl}(\mathsf E),
    \]
    where $\Ggamma^{\textrm{hor}}$ is as in \eqref{eq:gammahor} and $\Ggamma^{\textrm{ver}}$ is as in \eqref{eq:gammaver}. Property \ref{it:1} and the Hilbertianity of $(\mathcal{M}_+(\mathsf E), \HK, \mm)$ for any non-negative, finite Borel measure $\mm$ on $(\mathcal{M}_+(\mathsf E), \HK)$ are established in 
\cite[Corollary 4.17]{DscSod25}. Property \ref{it:2p} is proven in \cite[Proposition 4.7]{DscSod25}.
\end{example}

\begin{remark}[Extension to the Hilbertian case] We remark that, while the case of $\mathsf E$ being equal to a separable Hilbert space has been explicitly treated for the $L^2$-Kantorovich--Rubinstein--Wasserstein setting in Example \ref{ex:w2} in \cite[Section 6.3]{ForSavSod23}, this has not been done explicitly for the $L^2$-Hellinger and Hellinger--Kantorovich cases in Examples \ref{ex:he} and \ref{ex:hk}, respectively. However, exactly the same argument of \cite[Theorem 6.4]{ForSavSod23} can be used simply replacing $\mathsf W_2$ with $\mathsf{He}_2$ or $\HK$.
\end{remark}

We briefly comment on the general ideas used to establish properties \ref{it:1} and \ref{it:2p} in Examples \ref{ex:eu}-\ref{ex:hk}. 

Regarding property \ref{it:2p}, while its validity is straightforward in the Euclidean, Riemannian, and Hilbertian settings of Examples \ref{ex:eu}-\ref{ex:hi}, it requires a more delicate approach in the less concrete settings of Examples \ref{ex:he}-\ref{ex:hk}. In the latter cases, the property is obtained by exploiting the fine structure of geodesics in the underlying metric space. 

Property \ref{it:1} can be derived in all examples from the general criterion proven in \cite[Theorem 2.13]{ForSavSod23} and later extended in \cite[Theorem 2.12]{DscSod25}. This criterion shows that density in energy is achieved by approximating---in the Sobolev sense---functions of the form $\SS \ni x \mapsto \sfd(x,y)$ as the base point $y \in \SS$ varies.

Finally, as noted above, Hilbertianity follows directly from properties \ref{it:1} and \ref{it:2} (see \cite[Theorem 2.17]{ForSavSod23}).

We conclude this section mentioning the paper \cite{PasSod25} for the use of subalgebras of Lipschitz functions in metric BV spaces.

\section{Regularized least squares approximation of bounded functions}\label{sec:regbdfunc}
We show in this section how to recover some results of \cite{ForHeiSod25} originally obtained for the metric-measure space $(\prob_2(\mathsf E), \mathsf{W}_2, \mm)$ as in Example \ref{ex:w2}, within the more general setting of arbitrary metric-measure spaces. In Section \ref{sub:apprpce} we will work at the level of the pre-Cheeger energy defined in formula \eqref{eq:pce} as it is more common, see the Examples \ref{ex:eu}-\ref{ex:hk}, that one can compute it explicitly on relevant subalgebras $\AA$ of functions. 

We point out that, in case the pre-Cheeger energy is \emph{closable} on $\AA$ (that is, the $L^2(\SS, \mm)$ closure of the graph of $\pCE_{2, \mm}$ restricted to $\AA$ is still a graph) and $\AA$ is dense in energy as in \ref{it:1}, then the Cheeger energy, when restricted to $\AA$, coincides with the pre-Cheeger energy. Therefore, in the closable case, the procedure of Section \ref{sub:apprpce} can be carried out with $\CE_{2, \mm}$ instead of $\pCE_{2, \mm}$. This is indeed one of the few cases in which the Cheeger energy is explicitly computable, at least for the functions in $\AA$.
Closability obviously depends on the choice of the measure $\mm$ (besides, of course, on $(\SS, \sfd)$): it holds, for example, in Example \ref{ex:eu} when $\mm$ is the Lebesgue measure, in Example \ref{ex:ri} when $\mm$ is the Riemannian volume measure, and also in Examples \ref{ex:w2} and \ref{ex:hk} for specific choices of $\mm$, see the works \cite{Dsc25, DscSod25} for possible choices of such measures.

However, in general, one cannot hope for the closability property or the possibility to explicitly compute the Cheeger energy. This is however crucial to carry out the procedure of Section \ref{sub:apprpce} for the Cheeger energy. We can, however, work also at the level of the Cheeger energy, at the cost of introducing an approximation parameter, via a Lusin type procedure. This is the content of Section \ref{sub:apprce}.

\subsection{Direct pre-Cheeger approach}\label{sub:apprpce}
We consider a metric-measure space $(\SS, \sfd, \pp)$, with $\pp$ a probability, together with a unital subalgebra $\AA \subset \Lipb(\SS, \sfd)$ satisfying property \ref{it:2p} (with $\mm = \pp$) and additionally
\begin{pty}[Continuity]\label{it:3}
The map $\Ggamma$ in \ref{it:2p} takes values in $\rmC_b(\SS)$, the space of continuous and bounded functions on $(\SS, \sfd)$.
\end{pty}
Note that \ref{it:3} is satisfied in the Examples \ref{ex:eu}-\ref{ex:hk}.
We stress that, throughout this subsection, properties \ref{it:2p} (hence \ref{it:2}) and \ref{it:3} are the \emph{only} assumptions on $(\SS,\sfd,\pp)$ and $\AA$: in particular, neither the density property \ref{it:1} nor the infinitesimal Hilbertianity of $(\SS,\sfd,\pp)$ is required.
\medskip

We fix natural numbers $n, N \in \N$, a finite dimensional subspace $\mathcal{V}_n \subset \AA \subset H^{1,2}(\SS, \sfd, \pp)$ with $\dim(\mathcal{V}_n)=n$ and $N$ i.i.d.~points $x_1, \dots, x_N$ distributed on $\SS$ according to $\pp$; this means that we are considering an underlying probability space $(\Omega, \EE, \PP)$ and $N$ independent random variables $x_j : \Omega \to \SS$, $j=1, \dots, N$, each of them with law $\pp$. For instance, one can take $\Omega = \SS^N$ endowed with its Borel $\sigma$-algebra $\EE$ and with the product probability measure $\PP = \pp^{\otimes N}$, and $x_j(\omega) := \omega_j$ for $\omega = (\omega_1, \dots, \omega_N) \in \SS^N$. Consequently, every object depending on the points $x_1, \dots, x_N$ (such as the empirical measure $\pp_N$ below, the matrices of Definition~\ref{def:matrices}, and the approximants of Definition~\ref{def:someop}) is random, i.e., it depends on $\omega \in \Omega$ through $(x_1(\omega), \dots, x_N(\omega))$; we denote by $\mathbb E$ the expectation with respect to $\PP$. Accordingly we define the (random) empirical measure $\pp_N:= \frac{1}{N}\sum_{j=1}^N \delta_{x_j}$ and we fix $\lambda \ge 0$.

\medskip

Our aim is to estimate the $L^2(\SS, \pp)$ distance between a bounded Borel function $f$ which is known only on the support of $\pp_N$ and a suitable regularization of $f$ itself. We start with a definition.

\begin{definition}\label{def:someop} Given a function $f \in L^2(\SS, \pp_N)$, we define the functionals
\begin{align*}
\JJ_{N, f}^\lambda(g)&:= \|f-g\|^2_{L^2(\SS, \pp_N)} + \lambda \pCE_{2,\pp_N}(g), \quad &&g \in \Lip_b(\SS, \sfd),\\
\mathcal{N}_{N,f}(g) &:= \|f-g\|^2_{L^2(\SS, \pp_N)}, \quad &&g \in L^2(\SS, \pp_N).
\end{align*}
Both functionals admit a unique minimizer when restricted to $\mathcal{V}_n$ so that we can define
\begin{align*}
    S_{N}^{\lambda, n}(f):= \argmin_{g \in \mathcal{V}_n} \JJ_{N, f}^\lambda(g), \quad P_N^n (f):= \argmin_{g \in \mathcal{V}_n} \mathcal{N}_{N,f}(g).
\end{align*}
\end{definition}

In order to make the objects above computable, we fix a suitable basis of $\mathcal{V}_n$.

\begin{lemma}\label{le:db} There exists a finite family of functions $\mathcal{B}_n:=\{f_1, \dots, f_n\} \subset \mathcal{V}_n$ such that
\begin{enumerate}
    \item[(i)] $\mathcal{B}_n$ is a linear basis of $\mathcal{V}_n$;
    \item[(ii)] $\int_{\SS}f_i f_k \de \pp =0$ for every $i,k \in \{1, \dots, n\}$, $i \ne k$;
    \item[(iii)] $\int_{\SS} |f_i|^2 \de \pp =1$ for every $i=1,\dots, n$;
    \item[(iv)] $\pCE_{2,\pp}(f_i, f_k)=0$ for every $i,k \in \{1, \dots, n\}$, $i \ne k$.
\end{enumerate} 
\end{lemma}

\begin{proof} This follows by the simultaneous diagonalizations of the two bilinear forms
\[ (f,g) \mapsto \int_{\SS} fg \de \pp, \quad (f,g) \mapsto \pCE_{2, \pp}(f,g)\]
where, according to \eqref{eq:polar} and \ref{it:2p},
\[
\pCE_{2,\pp}(f,g) = \frac14 \Big( \pCE_{2,\pp}(f+g) - \pCE_{2,\pp}(f-g) \Big) = \int_\SS \Ggamma(f,g) \de \pp, \quad f,g \in \AA,
\]
is the symmetric bilinear form associated with $\pCE_{2,\pp}$ by polarization,
in the finite dimensional vector space $\mathcal{V}_n$. Note we are only using the bilinearity given by \ref{it:2}.
\end{proof}

In the next Lemma we make the expression of $S_{N}^{\lambda, n}(f)$  as in Definition \ref{def:someop} more explicit. To do that we first need to define two matrices.

\begin{definition}\label{def:matrices} Let $\{f_1, \dots, f_n\}$ be as in Lemma \ref{le:db}. We define the matrices $L_{N,n} \in \R^{N\times n}$ and $D_n \in \R^{n \times n}$ as 
\[ (L_{N,n})_{ji} := \frac{1}{\sqrt{N}}f_i(x_j), \quad (D_n)_{ik} = \pCE_{2, \pp_N}(f_i,f_k). \]
\end{definition}

Since the points $x_1, \dots, x_N$ are random, so are the matrices $L_{N,n}$ and $D_n$. The expected value (taken entrywise with respect to $\PP$) of the matrices $L_{N,n}^\top L_{N,n}$ and $D_n$ is easily computed:
\begin{equation}\label{eq:expv}
\mathbb{E} [L_{N,n}^\top L_{N,n}]= I_n \quad \text{and} \quad \mathbb{E}[D_n]= \Gamma_n,
\end{equation}
where $I_n$ is the identity matrix of order $n$ and $\Gamma_n \in \R^{n \times n}$ is a diagonal matrix with entries
\[ (\Gamma_n)_{ii}= \pCE_{2,\pp}(f_i), \quad i=1, \dots, n.\]

In the next Lemma, we show how to render the a priori abstract quantity $S_{N}^{\lambda, n}(f)$ computable. This constitutes the first step toward utilizing such quantities in real-world computer simulations, as demonstrated in \cite{ForHeiSod25}.

\begin{lemma}\label{le:fewthings} Let $\{f_1, \dots, f_n\}$ be as in Lemma \ref{le:db} and let $f \in L^2(\SS, \pp_N)$. Then $g=S_{N}^{\lambda, n}(f)$ (cf.~Definition \ref{def:someop}) if and only if
\[ g= \sum_{i=1}^n f_i w^\star_i,\]
where $w^\star \in \R^n$ is the unique minimizer in $\R^n$ of the functional $\JJ_{N, z^{f}}^{\lambda,n}: \R^n \to \R$ defined as
\begin{equation}\label{eq:funcrn}
 \JJ_{N, z^{f}}^{\lambda,n}(w):= |L_{N,n}(w-z^f)|^2_N + \lambda w^\top D_n w, \quad w \in \R^n, 
\end{equation}
where $|\cdot|_N$ denotes the Euclidean norm in $\R^N$ and $z^{f} \in \R^n$ is such that
\[ P^n_N f = \sum_{i=1}^n f_i z_i^{f}.\]
Equivalently $w^\star$ solves
\begin{align} \label{eq:wstar}
(L_{N,n}^\top L_{N,n} + \lambda D_n ) w^\star= L_{N,n}^\top L_{N,n} z^{f}.
\end{align}
In particular, $S_{N}^{\lambda, n}$ is a linear operator.
\end{lemma}

\begin{proof}
For every $g \in \mathcal{V}_n$, we can rewrite
\begin{align*}
\JJ_{N, f}^\lambda(g) &= \frac{1}{N}\sum_{j=1}^N |g(x_j)-f(x_j)|^2 + \lambda \pCE_{2, \pp_N}(g)
 \\
&= \frac{1}{N}\sum_{j=1}^N |g(x_j)-P^n_N f (x_j)|^2 + \lambda \pCE_{2, \pp_N}(g) + C(n,N,f)
\end{align*}
where $C(n,N,f)$ is a constant depending only on $f$, $n$ and $N$. Furthermore, upon defining
\[
\tilde{\JJ}_{N, f}^\lambda(g):= \frac{1}{N}\sum_{j=1}^N |g(x_j)-P^n_N f(x_j)|^2 + \lambda \pCE_{2, \pp_N}(g)
\]
we obtain that
\begin{align*}
\JJ_{N, f}^\lambda(g) = \tilde{\JJ}_{N, f}^\lambda(g) + C(n,N,f).
\end{align*}

Since $\tilde{\JJ}_{N, f}^\lambda$ and $\JJ_{N, f}^\lambda$ differ only by a constant, they have the same minimizer. For a given $g \in \mathcal{V}_n$ there exists a unique $w^g \in \R^n$ such that 
\[ g= \sum_{i=1}^n f_i w_i^g\]
and it is immediate to check that $\tilde{\JJ}_{N, f}^\lambda (g) = \JJ_{N, z^{f}}^{\lambda,n}(w^g)$. This concludes the proof.
\end{proof}

\begin{remark} \label{rem:wstar}
We note that~\eqref{eq:wstar} can equivalently be stated as
\[
(L_{N,n}^\top L_{N,n}+\lambda D_n)w^\star=y^f,
\]
where $y^f \in \mathbb{R}^n$ with
\[y^f_i=\frac{1}{N} \sum_{j=1}^N f(x_j) f_i(x_j), \qquad i =1,\dotsc,n.\]
In particular, we may obtain the solution $w^\star$ by solving a linear equation.
\end{remark}

Before stating the next result, we recall that the \emph{spectral norm} of a matrix $B \in \R^{n \times n}$ is defined as
\[
\|B\| := \max \left \{ |Bv| : v \in \R^n, \, |v| = 1 \right \} = \big( \lambda_{\max}(B^\top B) \big)^{1/2},
\]
where $|\cdot|$ is the Euclidean norm in $\R^n$ and, for a symmetric matrix $C \in \R^{n \times n}$, $\lambda_{\min}(C)$ and $\lambda_{\max}(C)$ denote the smallest and the largest eigenvalue of $C$, respectively. If $B$ is symmetric, then $\|B\| = \max\{ \lambda_{\max}(B), \lambda_{\max}(-B) \}$ is the largest modulus of the eigenvalues of $B$.

The next result is an auxiliary large deviation bound, which allows us to control the spectral behavior of the matrices introduced in Definition \ref{def:matrices}.

\begin{proposition} \label{prop:probbound} Let $L_{N,n}$ and $D_n$ be defined as in Definition \ref{def:matrices} and let $I_n$ and $\Gamma_n$ be as in \eqref{eq:expv}. Then 
\begin{align}\label{eq:lab1}
\PP ( \|L_{N,n}^\top L_{N,n} - I_n\| > \delta ) &\le 2n \exp \left \{ -N \frac{c_\delta}{ K_\lambda(n)}\right \} \quad \text{ for every } 0 < \delta <1
\end{align}
and, for every $0<\delta<1$,
\begin{align} \label{eq:lab2}
\mathbb{P} \left ( \|L_{N,n}^\top L_{N,n} + \lambda D_n -(I_n +\lambda \Gamma_n) \|> \delta \sigma_{\max,n}^{\lambda} \right ) &\le 2n \exp \left \{ -N \sigma_{\min,n}^{\lambda} \frac{c_\delta}{ K_\lambda(n)} \right \},
\end{align}
where 
\begin{align*}
 K_\lambda (n)&:= \sup_{x \in \SS} \sum_{i=1}^n \left(|f_i(x)|^2 +  \lambda  \Ggamma(f_i, f_i)(x) \right), \\
\sigma_{\min,n}^{\lambda} &:= 1+\lambda \min_{i=1, \dots, n} \pCE_{2, \pp}(f_i), \quad \sigma_{\max,n}^{\lambda} := 1+\lambda \max_{i=1, \dots, n} \pCE_{2, \pp}(f_i), \\
c_{\delta}&:= (1+\delta)\log(1+\delta)-\delta,
\end{align*}
and $\|\cdot\|$ denotes the spectral norm.
\end{proposition}
\begin{proof} The first bound follows precisely as in the proof of \cite[Theorem 1]{CohDavLev13, CohDavLev19}. The second one analogously: $L_{N,n}^\top L_{N,n} +\lambda D_n $ is the sum of $N$ symmetric and positive semi-definite matrices $X^j$, $j=1, \dots, N$, which are i.i.d.~copies of the matrix $X$ with entries
\[ X_{ik}:= X_{ik}^1+X_{ik}^2:= \frac{1}{N} f_i(x) f_k(x) +\frac{ \lambda}{N} \Ggamma(f_i, f_k)(x) , \quad i,k \in \{1, \dots, n\},
\]
with $x$ distributed according to $\pp$.

Indeed, since both $X^1$ and $X^2$ are symmetric and positive semi-definite matrices, so is $X$ and thus its spectral norm can be bounded by its trace; i.e., we have that
\[
\norm{X} \leq \frac{1}{N} \sum_{i=1}^n \left( |f_i(x)|^2+\lambda  \Ggamma(f_i,f_i)(x) \right) \qquad \text{for $\pp$-a.e.} \ x \in \SS.
\]
Invoking the definition of $ K_\lambda(n)$, this immediately yields the uniform bound
\[
\norm{X} \leq \frac{ K_\lambda(n)}{N}.
\]
In turn, for every $0<\delta<1$, the matrix Chernoff inequality, cf.~\cite[Theorem 1.1]{Tro12}, yields that
\begin{align*}
    \mathbb{P} \left ( \lambda_{\min}( L_{N,n}^\top L_{N,n} + \lambda D_n )-  \sigma_{\min,n}^{\lambda}  \le - \delta \sigma_{\min,n}^{\lambda} \right ) &= \mathbb{P} \left ( \lambda_{\min}(L_{N,n}^\top L_{N,n} +\lambda D_n) \le (1-\delta) \sigma_{\min,n}^{\lambda} \right )\\
    & \le n \left ( \frac{\mathrm{e}^{-\delta}}{(1-\delta)^{1-\delta}} \right )^{\frac{N \sigma_{\min,n}^{\lambda}}{ K_\lambda(n)}}\\
    & \le n \left ( \frac{\mathrm{e}^{\delta}}{(1+\delta)^{1+\delta}} \right )^{\frac{N \sigma_{\min,n}^{\lambda}}{ K_\lambda(n)}}
\end{align*}
where the first estimate follows from \cite[Theorem 1.1]{Tro12}, since the uniform bound $\norm{X} \le K_\lambda(n)/N$ obtained above provides the almost sure bound on the largest eigenvalue of each summand $X^j$ required there (the same holds for the first estimate in the next chain), and 
\begin{align*}
    \mathbb{P} \left ( \lambda_{\max}( L_{N,n}^\top L_{N,n} +\lambda D_n )- \sigma_{\max,n}^{\lambda}  \ge  \delta  \sigma_{\max,n}^{\lambda}  \right ) & = \mathbb{P} \left ( \lambda_{\max}(L_{N,n}^\top L_{N,n} +\lambda D_n ) \ge (1+\delta) \sigma_{\max,n}^{\lambda} \right )\\
    & \le n \left ( \frac{\mathrm{e}^{\delta}}{(1+\delta)^{1+\delta}} \right )^{\frac{N \sigma_{\max,n}^{\lambda} }{ K_\lambda(n)}}\\
    & \le n \left ( \frac{\mathrm{e}^{\delta}}{(1+\delta)^{1+\delta}} \right )^{\frac{N \sigma_{\min,n}^{\lambda}}{ K_\lambda(n)}}
\end{align*}
so that 
\[ \mathbb{P} \left ( \|  L_{N,n}^\top L_{N,n}   +   \lambda D_n   -(I_n +\lambda \Gamma_n) \|> \delta   \sigma_{\max,n}^{\lambda}    \right ) \le 2n \exp \left \{ -N   \sigma_{\min,n}^{\lambda}    \frac{c_\delta}{ K_\lambda(n)} \right \}.\]
\end{proof}

The condition~\eqref{eq:thecond} below ensures that the probabilities in Proposition~\ref{prop:probbound} are small. More precisely, let the subspace $\mathcal{V}_n$ (hence its dimension $n$ and the basis $\{f_1, \dots, f_n\}$ of Lemma~\ref{le:db}), the parameter $\lambda \ge 0$, and an order $r>0$ be fixed, so that the quantities $K_\lambda(n)$, $\sigma_{\min,n}^{\lambda}$, and $\sigma_{\max,n}^{\lambda}$ of Proposition~\ref{prop:probbound} are fixed as well. Since these quantities do not depend on $N$, while $N/\log(N) \to +\infty$ as $N \to \infty$, condition~\eqref{eq:thecond} is satisfied by every sufficiently large number $N$ of point evaluations; as we show below, it guarantees that the sum of the probabilities in \eqref{eq:lab1} and \eqref{eq:lab2}, for suitable choices of $\delta$, is smaller than $N^{-r}$. The required condition reads as
\begin{equation}\label{eq:thecond}
    \frac{N}{\log(N)} \ge \frac{(1+r) K_\lambda(n)}{  \sigma_{\min,n}^{\lambda}   c_{\frac{1}{2  \sigma_{\max,n}^{\lambda}   }}}, \quad N \ge 2,
\end{equation}
which further implies that
\[ \frac{N}{\log(N)} \ge \frac{(1+r) K_\lambda(n)}{c_{1/2}}.\]

Indeed, given~\eqref{eq:thecond} and choosing $\delta=1/2$ in \eqref{eq:lab1} as well as $\delta=(2\sigma_{\max,n}^{\lambda})^{-1}$ in \eqref{eq:lab2} leads to
\begin{multline} \label{eq:thecondcons}
\PP \left ( \|(  L_{N,n}^\top L_{N,n}   +  \lambda D_n  )-(I_n +\lambda \Gamma_n)\| > \frac{1}{2} \right ) + \PP \left ( \|  L_{N,n}^\top L_{N,n}   - I_n\| > \frac{1}{2} \right ) \\ \le 4\frac{n}{N} N^{-r}
\le 4 \frac{ K_\lambda(n)}{N} N^{-r} \le \frac{4c_{1/2}}{\log(N)(1+r)} N^{-r}  \le \frac{4c_{1/2}}{\log(2)} N^{-r} \le N^{-r},
\end{multline}
where we have used that $ K_\lambda(n) \ge n$ and that $N \ge 2$.

With the bound \eqref{eq:thecondcons} at our disposal we can reproduce two of the main results of \cite{ForHeiSod25}.

In the first result the data are noiseless samples of the function $f$ we intend to recover. 

\begin{theorem}\label{teo:cohen2} Let $(\SS,\sfd,\pp)$ and the unital subalgebra $\AA \subset \Lipb(\SS,\sfd)$ satisfy properties \ref{it:2p} (with $\mm = \pp$) and \ref{it:3}. Let $M>0$ and let $f: \SS \to [-M,M]$ be a measurable function and let $f^{\star}$ be defined as
\[ f^\star := -M \vee S_N^{\lambda,n}(f) \wedge M,\]
where $S_N^{\lambda,n}$ is as in Definition \ref{def:someop}; let $r>0$ be given and assume that $n,N \in \N$ satisfy \eqref{eq:thecond}. Then
\begin{multline*}
    \mathbb{E} \left [ \|f-f^\star\|_{L^2(\SS, \pp)}^2\right ] \\
    \le 2 e(f,n) \left ( 1 + \frac{ c_{1/2}}{\log(N)(1+r)(1/2+\lambda   \mu_{\min,n}  )^2} \right ) + 4 \lambda \pCE_{2, \pp}(P_n f) +  4  M^2 N^{-r},
\end{multline*}
where $P_n f$ is the ${L^2(\SS, \pp)}$-orthogonal projection of $f$ onto $\mathcal{V}_n$  and 
\begin{align*}
e(f,n) &:= \|P_n f-f\|^2_{L^2(\SS, \pp)},\\
  \mu_{\min,n}  &:= \min_{i=1, \dots, n} \pCE_{2, \pp}(f_i). 
\end{align*}
\end{theorem}

\begin{proof}
Recall that we have fixed an underlying probability space $(\Omega, \EE, \PP)$; we consider a partition of $\Omega$ into the two sets
\[ \Omega_+:= \left \{ \omega \in \Omega : \|(  L_{N,n}^\top L_{N,n}   +  \lambda D_n  )-(I_n +\lambda \Gamma_n)\| \le \frac{1}{2} \wedge \|  L_{N,n}^\top L_{N,n}   - I_n\| \le \frac{1}{2} \right \}\]
 and its complement $\Omega_-:= \Omega\setminus \Omega_+$. Notice that \eqref{eq:thecondcons} gives $\PP(\Omega_-) \le N^{-r}$. Since $|f - f^\star| \le 2M$ in $\SS$, we have $\|f-f^\star\|_{L^2(\SS, \pp)}^2 \le 4M^2$ and thus
\[ \mathbb{E} \left [ \|f-f^\star\|_{L^2(\SS, \pp)}^2\right ] \le \int_{\Omega_+} \|f-f^\star\|_{L^2(\SS, \pp)}^2 \de \PP + 4 M^2 N^{-r}.\]
We are left to estimate 

\[ \int_{\Omega_+} \|f-f^\star\|_{L^2(\SS, \pp)}^2 \de \PP \le \int_{\Omega_+} \|f-S_N^{\lambda,n}(f)\|_{L^2(\SS, \pp)}^2 \de \PP. \]

Since by Lemma \ref{le:fewthings} the operator $S_N^{\lambda,n}$ is linear, we have
\[ f-S_N^{\lambda,n}(f) = (I-S_{N}^{\lambda, n})h + P_n f -S_{N}^{\lambda, n} (P_n f), \]
where $h= f-P_nf$.

Observe that for $\xi \in \mathcal{V}_n$ and a draw in $\Omega_+$ it holds
\begin{align} \label{eq:Ineq1} 
\int_{\SS} |\xi|^2 \de \pp \le 2 \int_{\SS} |\xi|^2 \de \pp_N.
\end{align}

By the very definition of $S_{N}^{\lambda, n}$ we further find that
\begin{align*}
 \|S_{N}^{\lambda, n}(G)- G\|^2_{L^2(\SS, \pp_N)} &\le \|S_{N}^{\lambda, n}(G)- G\|^2_{L^2(\SS, \pp_N)} + \lambda \pCE_{2,\pp_N}(S_{N}^{\lambda, n}(G)) \\
 &\le \|G- G\|^2_{L^2(\SS, \pp_N)} + \lambda \pCE_{2, \pp_N}(G) \\
 &= \lambda \pCE_{2, \pp_N}(G)
\end{align*}
for any $G \in \mathcal{V}_n$. Combining this inequality with~\eqref{eq:Ineq1} yields
\begin{align*}
\int_{\Omega_+} \|P_n f -S_{N}^{\lambda, n} (P_n f)\|_{L^2(\SS, \pp)}^2 \de \PP &\le 2 \int_{\Omega_+} \|P_n f -S_{N}^{\lambda, n} (P_n f)\|_{L^2(\SS, \pp_N)}^2   \de \PP\\
& \le 2\lambda \int_{\Omega_+} \pCE_{2, \pp_N}(P_n f) \de \PP\\
& \le 2 \lambda \, \mathbb E \left [ \pCE_{2, \pp_N}(P_n f) \right ] = 2 \lambda \pCE_{2, \pp}(P_n f),
\end{align*}
where in the last equality we used that $P_n f$ is deterministic and that each $x_j$ has law $\pp$, so that $\mathbb E[\pCE_{2,\pp_N}(P_n f)] = \frac1N \sum_{j=1}^N \mathbb E[(\lip_\sfd P_n f(x_j))^2] = \pCE_{2,\pp}(P_n f)$.
It only remains to estimate
\[ \int_{\Omega_+} \| (I-S_{N}^{\lambda, n})h\|_{L^2(\SS, \pp)}^2 \de \PP.\]
Since $h$ is $L^2(\SS, \pp)$-orthogonal to $\mathcal{V}_n$ and $S_{N}^{\lambda, n} h$ belongs to $\mathcal{V}_n$, we have that 
\[ \| (I-S_{N}^{\lambda, n})h\|_{L^2(\SS, \pp)}^2 = \|h\|^2_{L^2(\SS, \pp)} + \|S_{N}^{\lambda, n} h\|^2_{L^2(\SS, \pp)}. \]
Precisely as in \cite{CohDavLev13, CohDavLev19} one then shows that
\[ \int_{\Omega_+}\|S_{N}^{\lambda, n} h\|^2_{L^2(\SS, \pp)} \de \mathbb{P} \le \alpha_{\lambda,n}^2 \frac{ K_\lambda(n)}{N} \|h\|^2_{L^2(\SS, \pp)}, \]
where
\begin{equation}\label{eq:alphaln}
\alpha_{\lambda,n}:=\frac{1}{\frac{1}{2} + \lambda \min_{i=1, \dots, n} \pCE_{2, \pp} (f_i)}.
\end{equation}

Putting everything together we get
\begin{multline*}
\mathbb{E} \left [ \|f-f^\star\|_{L^2(\SS, \pp)}^2\right ] \\
\begin{aligned}
&\le \int_{\Omega_+} \|f-f^\star\|_{L^2(\SS, \pp)}^2 \de \PP + 4 M^2 N^{-r}\\
&\le  2  \int_{\Omega_+} \| (I-S_{N}^{\lambda, n})h\|_{L^2(\SS, \pp)}^2 \de \PP + 4  \lambda \pCE_{2, \pp}(P_n f) + 4 M^2 N^{-r}\\
&\le  2  \|h\|^2_{L^2(\SS, \pp)} +  2\alpha_{\lambda,n}^2  \frac{ K_\lambda(n)}{N} \|h\|^2_{L^2(\SS, \pp)} + 4 \lambda \pCE_{2, \pp}(P_n f) + 4 M^2 N^{-r}.
\end{aligned}
\end{multline*}
Finally, recalling the definition of $h$ and employing~\eqref{eq:thecond} yield that
\begin{align*}
\mathbb{E} \left [ \|f-f^\star\|_{L^2(\SS, \pp)}^2\right ] &\le \|P_n f-f\|^2_{L^2(\SS, \pp)} \left (  2  + \frac{  2\alpha_{\lambda,n}^2  c_{1/2}}{\log(N)(1+r)} \right ) \\
& \quad + 4  \lambda \pCE_{2, \pp}(P_n f) +  4M^2 N^{-r}.
\end{align*}
\end{proof}

We assume now to know the target function $f$ at the (random) points $x_1, \dots, x_N$ only up to some noise. In particular, we assume that we have at our disposal the noisy values
\begin{equation}\label{eq:feta}
  \widetilde{f}  (x_j):= f(x_j)+\eta_j, \quad j=1, \dots, N,
\end{equation}
where $(\eta_j)_{j=1}^N$ are centered i.i.d.~random variables independent of $x_1, \dots, x_N$ with variance bounded by $\sigma^2>0$. Notice that the (random) function $  \widetilde{f}  $ belongs to $L^2(\SS, \pp_N)$.

\begin{theorem}\label{teo:cohen3}  Let $(\SS,\sfd,\pp)$ and the unital subalgebra $\AA \subset \Lipb(\SS,\sfd)$ satisfy properties \ref{it:2p} (with $\mm = \pp$) and \ref{it:3}. Let $M>0$, let $f: \SS \to [-M,M]$ be a measurable function and let $  \widetilde{f}^\star   $ be defined as
\[   \widetilde{f}^\star    := -M \vee S_N^{\lambda,n}(  \widetilde{f}  ) \wedge M,\]
where $S_N^{\lambda,n}$ is as in Definition \ref{def:someop} and $  \widetilde{f}  $ is as in \eqref{eq:feta} based on the noisy data $(\widetilde f(x_j))_{j=1}^N$ whose centered i.i.d.~errors $(\eta_j)_{j=1}^N$ are independent of $x_1, \dots, x_N$ and have variance bounded by $\sigma^2>0$. Let $r>0$ be given and assume that $n,N \in \N$ satisfy \eqref{eq:thecond}. Then
\begin{align*}
\mathbb{E} \left [ \|f-  \widetilde{f}^\star   \|_{L^2(\SS, \pp)}^2\right ] &\le 2 e(f,n) \left ( 1 + \frac{ c_{1/2}}{\log(N)(1+r)(1/2+\lambda   \mu_{\min,n}  )^2} \right ) + 8\lambda \pCE_{2, \pp}(P_n f) + \\
& \quad + 4\frac{\sigma^2}{(1/2+\lambda   \mu_{\min,n}  )^2}\frac{n}{N}+ 4  M^2 N^{-r},
\end{align*}
where $e(f,n)$ and $  \mu_{\min,n}  $ are as in Theorem \ref{teo:cohen2}.
\end{theorem}

\begin{proof}
The proof follows the one of Theorem \ref{teo:cohen2} and it is also inspired by \cite[Theorem 3]{CohDavLev13, CohDavLev19}.

We split again $\Omega$ into the sets $\Omega_+$ and $\Omega_-$. Performing the same calculations as before, we find that
\[ \mathbb{E} \left [ \|f-  \widetilde{f}^\star   \|_{L^2(\SS, \pp)}^2\right ] \le \int_{\Omega_+} \|f-S_N^{\lambda,n}(  \widetilde{f}  )\|_{L^2(\SS, \pp)}^2 \de \PP + 4 M^2 N^{-r}.\]

We can write
\[f - S_N^{\lambda,n}(  \widetilde{f}  ) = (I-S_{N}^{\lambda, n})(h + P_n f) - S_{N}^{\lambda, n}(\eta)\]
where $h= f-P_nf$, $P_n$ is the orthogonal $L^2(\SS, \pp)$ projection on $\mathcal{V}_n$ and $\eta \in L^2(\SS, \pp_N)$ is the (random) function taking values $\eta(x_j):=\eta_j$, $j=1, \dots, N$.

Following precisely the same steps as in the previous proof, we obtain the bound
\begin{align*}
\int_{\Omega_+} \|f-S_N^{\lambda,n}(  \widetilde{f}  )\|_{L^2(\SS, \pp)}^2 \de \PP &\le 2 \|h\|^2_{L^2(\SS, \pp)}+8\lambda \pCE_{2, \pp}(P_n f)\\
&\quad + \int_{\Omega_+} \left ( 2\|S_N^{\lambda,n}(h)\|^2_{L^2(\SS, \pp)} + 4\|S_N^{\lambda,n}(\mathcal{\eta})\|^2_{L^2(\SS, \pp)} \right )\de \PP.    
\end{align*}

Repeating the computations of \cite[Theorem 3]{CohDavLev13, CohDavLev19} then leads to
\[\int_{\Omega_+} \left ( 2\|S_N^{\lambda,n}(h)\|^2_{L^2(\SS, \pp)} + 4\|S_N^{\lambda,n}(\mathcal{\eta})\|^2_{L^2(\SS, \pp)} \right )\de \PP \le 2\alpha_{\lambda, n}^2\frac{K_\lambda(n)}{N} \|h\|^2_{L^2(\SS, \pp)} + 4\alpha_{\lambda,n}^2\sigma^2 \frac{n}{N}, \]
where $\alpha_{\lambda, n}$ is as in \eqref{eq:alphaln}; here we also used that $\mathbb E[|y^\eta|^2] \le \sigma^2 n/N$, with $y^\eta$ as in Remark~\ref{rem:wstar}, which follows from the fact that the errors $\eta_j$ are centered, independent, and independent of $x_1, \dots, x_N$.
Combining all the estimates and doing some simple manipulations finally lead to the claimed upper bound (recall that $\alpha_{\lambda,n} = (1/2 + \lambda \mu_{\min,n})^{-1}$).
\end{proof}

\begin{remark}
In practical applications, it is not always straightforward to determine the approximation subspace $\mathcal{V}_n$ that optimizes $e(f,n)$. A possible heuristic remedy, explored in \cite[Section 6.4]{ForHeiSod25}, is the employment of neural networks.
\end{remark}

\subsection{Approximated Cheeger energy approach}\label{sub:apprce}
In this subsection, we want to reproduce the results of Section \ref{sub:apprpce} but working directly at the level of the Cheeger energy. In general, we cannot assume that condition \ref{it:3} holds with $\pCE$ being replaced by $\CE$ (unless, for example, we are in the special closable case discussed at the beginning of Section \ref{sec:regbdfunc}).

Still, whenever the metric-measure space is infinitesimally Hilbertian, the Cheeger energy can be represented in terms of a bilinear map, in the following sense. We consider a metric-measure space $(\SS, \sfd, \pp)$ with $\pp$ a probability and, unlike in Section \ref{sub:apprpce}, we make the following standing assumption.
{\addtocounter{pty}{-1}\renewcommand{\thepty}{H}\def\theHpty{H}%
\begin{pty}[Infinitesimal Hilbertianity]\label{it:H}
The metric-measure space $(\SS, \sfd, \pp)$ is infinitesimally Hilbertian.
\end{pty}}
\noindent We stress that \ref{it:H} is the \emph{only} assumption of this subsection: none of the properties \ref{it:1}, \ref{it:2}, \ref{it:2p}, \ref{it:3} is required, and no subalgebra $\AA$ of Lipschitz functions is involved. By \ref{it:H}, there exists a symmetric bilinear map $\Ggamma: H^{1,2}(\SS, \sfd, \pp) \times H^{1,2}(\SS, \sfd, \pp) \to L^1(\SS, \pp)$ such that
\[
\CE_{2, \pp}(f,h) = \int_{\SS} \Ggamma(f,h) \de \pp, \quad \Ggamma(f,f) \ge 0 \text{ $\pp$-a.e.}  \quad \text{ for every } f,h \in H^{1,2}(\SS, \sfd, \pp).
\]
Here, we are denoting by $\CE_{2,\pp}(\cdot, \cdot)$ the symmetric bilinear form associated by polarization with $\CE_{2, \pp}$, which is well defined since $\CE_{2,\pp}$ satisfies the parallelogram identity because of the Hilbertianity of $H^{1,2}(\SS, \sfd, \pp)$. The existence of the map $\Ggamma$ as above can be deduced, for example, from \cite[Theorem 4.3.3]{GigPas20}.

As in Section \ref{sub:apprpce}, we fix natural numbers $ N \in \N$ and $N$ i.i.d.~points $x_1, \dots, x_N$ distributed on $\SS$ according to $\pp$, i.e., $N$ independent random variables $x_j : \Omega \to \SS$ with law $\pp$, defined on a probability space $(\Omega, \EE, \PP)$ (see Section~\ref{sub:apprpce} for more details). Accordingly, we define the (random) empirical measure $\pp_N:= \frac{1}{N}\sum_{j=1}^N \delta_{x_j}$ and we fix $\lambda \ge 0$.

\medskip

We consider a family of $n$-dimensional subspaces $(\mathcal{V}_n)_n \subset \Lip_b(\SS, \sfd) \subset H^{1,2}(\SS, \sfd, \pp)$; arguing as in Lemma \ref{le:db}, for any $n \in \N$, we can find a family of functions $\mathcal{B}_n:=\{f_1, \dots, f_n\}$ such that 
\begin{enumerate}
    \item[\mylabel{it:a}{(a)}]  $\mathcal{B}_n$ is a linear basis of $\mathcal{V}_n$;
    \item[\mylabel{it:b}{(b)}] $\int_{\SS}f_i f_k \de \pp =0$ for every $i,k \in \{1, \dots, n\}$, $i \ne k$;
    \item[\mylabel{it:c}{(c)}]  $\int_{\SS} |f_i|^2 \de \pp =1$ for every $i=1,\dots, n$;
    \item[\mylabel{it:d}{(d)}]  $\CE_{2,\pp}(f_i, f_k)=0$ for every $i,k \in \{1, \dots, n\}$, $i \ne k$.
\end{enumerate} 
We denote 
\[
\mathcal B := \bigcup_n \mathcal B_n \subset \Lipb(\SS,\sfd).
\]

Instead of assuming directly that $\Ggamma$ takes values in the space of continuous functions, we can produce an approximation of it with this property, at the price of introducing an additional error parameter.

\begin{lemma}\label{le:approx} For every $\eps \in (0,1)$ there exist a Borel subset $A_\eps \subset \SS$ with $\pp(A_\eps) \le \eps$ and a symmetric map $\Ggamma_\eps: \mathcal B \times \mathcal B \to \rmC_b(\SS)$ such that
\[ \Ggamma_\eps(f,h)=\Ggamma(f,h) \text{ in } \SS \setminus A_\eps  \text{ and } \int_\SS |\Ggamma_\eps(f,h) - \Ggamma(f,h) | \de \pp < \eps  \quad \text{ for every } f,h \in \mathcal B.  \]
\end{lemma}
\begin{proof} 
Fix $(f,h) \in \mathcal B \times \mathcal B$ and $\delta>0$. For simplicity, we set $u:= \Ggamma(f,h)$; since $u \in L^1(\SS, \pp)$, we can find $M_\delta>1$ such that 
\begin{equation*}
    \int_{\{ |u| > M_\delta \}} |u| \de \pp < \delta/6.
\end{equation*}
Therefore, by Markov's inequality, we have
\[
M_\delta \pp( \{ |u|>M_\delta \} ) \le \delta/6.
\]
Set $u_\delta := -M_\delta \vee u \wedge M_\delta$. By Lusin's theorem applied to the finite Borel measure
\[
\mm := (|u|+1)\pp
\]
and the bounded function $u_\delta$, there exists a closed set $E_\delta \subset \SS$ such that $u_\delta$ is continuous on $E_\delta$ and $\mm(\SS \setminus E_\delta)< \delta/(3M_\delta)$. Let us set $B_\delta := (\SS \setminus E_{\delta}) \cup \{ |u|>M_\delta\}$. We compute
\begin{align*}
    \pp(B_\delta) \le \pp( \SS \setminus E_{\delta}) + \pp (\{ |u|>M_\delta\}) \le \mm( \SS \setminus E_{\delta}) + \delta/(6 M_\delta) < \delta.
\end{align*}
Since $E_\delta$ is closed, by Tietze's extension theorem, there exists a continuous function $v_\delta$ such that $v_\delta=u_\delta$ on $E_\delta$ and $\|v_\delta\|_\infty \le M_\delta$. On $\SS \setminus B_{\delta} = E_{\delta} \setminus \{ |u|>M_\delta\}$, we have $v_{\delta}=u_\delta$ and $u_\delta=u$ so that $v_{\delta}=u$ on $\SS \setminus B_{\delta}$. We compute
\begin{align*}
 \int_\SS |u-v_{\delta}| \de \pp &= \int_{B_{\delta}} |u-v_{\delta}| \de \pp 
    \\
    & \le \int_{\SS \setminus E_\delta} |u| \de \pp + M_\delta \pp (\SS \setminus E_\delta) +  \int_{\{ |u| > M_\delta\}} |u| \de \pp + M_\delta \pp (\{ |u| > M_\delta\})  
    \\
    & \le \mm( \SS \setminus E_\delta)(1+M_\delta) + \delta/6 + \delta/6
    \\
    & \le \delta/(3M_\delta)(1+M_\delta) + \delta/3
    \\
    & \le \delta.
\end{align*}
We have shown that for every $\delta>0$ and every $(f,h) \in \mathcal{B} \times \mathcal{B}$ there exists $v_{\delta}^{f,h} \in \rmC_b(\SS)$ and a Borel subset $B_\delta^{f,h} \subset \SS$ with $\pp(B_\delta^{f,h}) < \delta$ such that
\[
v_\delta^{f,h} = \Ggamma(f,h) \text{ in } \SS \setminus B_\delta^{f,h}, \quad \int_{\SS} |v_\delta^{f,h} - \Ggamma(f,h) | \de \pp < \delta.
\]
Let $\eps>0$ be fixed and let $(f_l)_l$ be an enumeration of $\mathcal B$; for every $l,m \in \N$ with $m \le l$ set 
$\delta_{l,m}:= \eps/2^{l+m}$ and
\[ \Ggamma_\eps(f_l, f_m) := v_{\delta_{l,m}}^{f_l, f_m}, \quad A_\eps := \bigcup_{l,m} B_{\delta_{l,m}}^{f_l, f_m}.\]
Then $\pp( A_\eps) \le \eps$, $\Ggamma_\eps(f_l, f_m) \in \rmC_b(\SS)$ and, for every $l,m \in \N$ with $m \le l$, on $\SS \setminus A_\eps \subset \SS \setminus B_{\delta_{l,m}}^{f_l, f_m}$, we have $\Ggamma_{\eps}(f_l, f_m)=v_{\delta_{l,m}}^{f_l, f_m} =\Ggamma(f_l, f_m)$ and
\[
\int_\SS |\Ggamma_\eps(f_l, f_m) - \Ggamma(f_l, f_m)| \de \pp = \int_{\SS} |v_{\delta_{l,m}}^{f_l, f_m} - \Ggamma(f_l, f_m) | \de \pp < \delta_{l,m} \le \eps.
\]
When $l,m \in \N$ with $m >l$, we simply set
\[
\Ggamma_\eps(f_l, f_m) := \Ggamma_\eps(f_m, f_l).
\]
By symmetry of $\Ggamma$, we conclude the proof.
\end{proof}

For every $\eps \in (0,1)$, and every non-negative, finite Borel measure $\mm$ on $(\SS, \sfd)$,  we define a map $\CE^\eps_{2, \mm}: \mathcal B \times \mathcal B \to \R$ as
\[
\CE^\eps_{2, \mm}(f,h) := \int_\SS \Ggamma_\eps(f,h) \de \mm \quad f,h \in \mathcal B \times \mathcal B,
\]
where $\Ggamma_\eps$ is as in Lemma \ref{le:approx}; for every $n \in \N$, we extend $\CE^\eps_{2, \mm}$ by bilinearity to $\mathcal V_n \times \mathcal V_n$, and we set $\CE^\eps_{2, \mm}(g) := \CE^\eps_{2, \mm}(g,g)$ for $g \in \mathcal V_n$.

\medskip

We now fix $n \in \N$ (corresponding to the subspace $\mathcal{V}_n$) and a tolerance $\gamma \in (0,1)$ for the rest of the subsection: the latter parameter will quantify the error we make in approximating the Cheeger energy $\CE_{2, \pp}$ with the above approximation $\CE_{2, \pp}^\eps$. We choose $\eps \in (0,\gamma)$ (clearly depending also on $\gamma$ and $n$) such that 
\begin{equation}\label{eq:delta}
\left ( \sum_{i=1}^n \sum_{k=1}^n \left | \CE_{2,\pp}^{\eps}(f_i, f_k)  -\CE_{2, \pp}(f_i, f_k) \right |^2 \right )^{1/2} < \gamma.
\end{equation}
The possibility of doing so follows of course from Lemma \ref{le:approx}.

\medskip

We define the analogous matrices as in Definition \ref{def:matrices}, additionally depending on the parameter $\eps$ as above.

\begin{definition}\label{def:matrices2} Let $\{f_1, \dots, f_n\}$ be as in \ref{it:a}-\ref{it:d}. We define the matrices $ L_{N,n} \in \R^{N\times n}$ and $I_n, D_{N,n}^\eps, \Gamma_n^\eps, \Gamma_n \in \R^{n \times n}$ as
\[ (L_{N,n})_{ji} := \frac{1}{\sqrt{N}}f_i(x_j), \quad (D_{N,n}^\eps)_{ik} = \int_\SS \Ggamma_\eps(f_i, f_k) \de \pp_N, \quad (\Gamma_n^\eps)_{ik}:= \CE^{\eps}_{2, \pp}(f_i,f_k) \]
with $I_n$ the identity matrix in $\R^{n\times n}$ and $\Gamma_n$ the diagonal matrix with entries 
\[ (\Gamma_n)_{ii} = \CE_{2, \pp}(f_i, f_i), \quad i =1, \dots, n.\]
\end{definition}

The expected value of the matrices $L_{N,n}^\top L_{N,n}$ and $D_{N,n}^\eps$ is easily computed:
\begin{equation}\label{eq:expv2}
\mathbb{E} [L_{N,n}^\top L_{N,n}]= I_n \quad \text{and} \quad \mathbb{E}[D_{N,n}^\eps]= \Gamma_n^\eps;
\end{equation}
By \eqref{eq:delta}, we see that $\|\Gamma_n^\eps - \Gamma_n\| < \gamma$ so that
\[ \|D_{N,n}^\eps - \Gamma_n \| \le \|D_{N,n}^\eps - \Gamma_n^\eps\| + \gamma.\]

\begin{proposition} \label{prop:probbound2} Let $I_n$, $L_{N,n}$, $D_{N,n}^\eps$ and $\Gamma_n^\eps$ be defined as in Definition \ref{def:matrices2}. Then
\begin{align}\label{eq:lab12}
\PP ( \|L_{N,n}^\top L_{N,n} - I_n\| > \delta ) &\le 2n \exp \left \{ -N \frac{c_\delta}{ K_\lambda^\eps(n)}\right \} \quad \text{ for every } 0 < \delta <1
\end{align}
and, for every $0<\delta<1$, we have
\begin{align} \label{eq:lab22}
\mathbb{P} \left ( \|L_{N,n}^\top L_{N,n} + \lambda D_{N,n}^\eps -(I_n +\lambda \Gamma_n^\eps) \|> \delta \sigma_{\max,n}^{\lambda,\eps} \right ) &\le 2n \exp \left \{ -N \sigma_{\min,n}^{\lambda,\eps} \frac{c_\delta}{ K_\lambda^\eps(n)} \right \} + N\eps,
\end{align}
where 
\begin{align*}
 K_\lambda^\eps (n)&:= \sup_{x \in \SS \setminus A_\eps} \sum_{i=1}^n \left(|f_i(x)|^2 +  \lambda \Ggamma_\eps(f_i, f_i)(x) \right ), \\
\sigma_{\min,n}^{\lambda,\eps} &:= 1+\lambda \min_{i=1, \dots, n} \CE^{\eps}_{2, \pp}(f_i,f_i), \quad \sigma_{\max,n}^{\lambda,\eps} := 1+\lambda \max_{i=1, \dots, n} \CE^{\eps}_{2, \pp}(f_i,f_i), \\
c_{\delta}&:= (1+\delta)\log(1+\delta)-\delta,
\end{align*}
and $\|\cdot\|$ denotes the spectral norm.
\end{proposition}
\begin{proof} The first bound follows precisely as in the proof of \cite[Theorem 1]{CohDavLev13, CohDavLev19}. For the second one, we proceed as follows: we split $\Omega$ in two parts $\Omega_+$ and $\Omega_-$ defined as
\[
\Omega_+:= \{ \omega \in \Omega : x_j(\omega) \in \SS \setminus A_\eps \text{ for every } j=1, \dots, N\}, \quad \Omega_- := \Omega \setminus \Omega_+.
\]

Consider the new probability space $(\Omega_+, \mathcal E_+, \mathbb P_+)$, where $\mathcal E_+$ is the trace of $\mathcal E$ on $\Omega_+$ and $\mathbb P_+$ is simply the normalized restriction of $\mathbb P$ to $\Omega_+$. For every $E \in \mathcal E$, we clearly have
\[
\mathbb P(E) \le \mathbb P_+(E \cap \Omega_+) \mathbb P(\Omega_+) + \mathbb P (\Omega_-) \le \mathbb P_+(E \cap \Omega_+) + N\eps.
\]
On $\Omega_+$, exactly the same estimates of Proposition \ref{prop:probbound} can be applied noticing that, there, $L_{N,n}^\top L_{N,n} +\lambda D_{N,n}^\eps $ is the sum of $N$ independent symmetric and positive semi-definite matrices $X^j$, $j=1, \dots, N$, with entries
\[ X_{ik}^j:= \frac{1}{N} f_i(x_j) f_k(x_j) +\frac{ \lambda}{N} \Ggamma_\eps(f_i, f_k)(x_j) , \quad i,k \in \{1, \dots, n\}.
\]
Indeed on $\Omega_+$, we have that $\Ggamma_\eps(f_i, f_k)(x_j)=\Ggamma(f_i, f_k)(x_j)$ and the latter matrix is symmetric and positive semi-definite.
\end{proof}

Once that we have the bounds \eqref{eq:lab12}, \eqref{eq:lab22} at our disposal, we can derive the analogous results as in Theorems \ref{teo:cohen2} and \ref{teo:cohen3}, with the estimates containing the error introduced with the approximation parameter $\eps \in (0,1)$.

\begin{theorem}\label{teo:cohen2eps} Let $(\SS,\sfd,\pp)$ satisfy property \ref{it:H}. Let $M>0$ and let $f: \SS \to [-M,M]$ be a measurable function and let $f^{\star}$ be defined as
\[ f^\star := -M \vee S_N^{\lambda,n}(f) \wedge M,\]
where $S_N^{\lambda,n}$ is as in Definition \ref{def:someop} with $\CE^{\eps}_{2, \pp_N}$ in place of $\pCE_{2, \pp_N}$; let $r>0$ be given and assume that $n,N \in \N$ satisfy \eqref{eq:thecond}. Then
\begin{multline*}
    \mathbb{E} \left [ \|f-f^\star\|_{L^2(\SS, \pp)}^2\right ] \\
    \le 2 e(f,n) \left ( 1 + \frac{ c_{1/2}}{\log(N)(1+r)(1/2+\lambda   \mu^\eps_{\min,n}  )^2} \right ) + 4 \lambda \CE_{2, \pp}(P_n f) +  4  M^2 (N^{-r}+N\eps),
\end{multline*}
where $P_n f$ is the ${L^2(\SS, \pp)}$-orthogonal projection of $f$ onto $\mathcal{V}_n$  and
\begin{align*}
e(f,n) &:= \|P_n f-f\|^2_{L^2(\SS, \pp)},\\
  \mu^\eps_{\min,n}  &:= \min_{i=1, \dots, n} \CE^{\eps}_{2, \pp}(f_i,f_i).
\end{align*}
\end{theorem}

In case of noisy data, we can do the procedure as in Section \ref{sub:apprpce}: assume to know the target function $f$ at the (random) points $x_1, \dots, x_N$ only up to some noise. In particular, we assume that we have at our disposal the noisy values
\begin{equation}\label{eq:feta2}
  \widetilde{f}  (x_j):= f(x_j)+\eta_j, \quad j=1, \dots, N,
\end{equation}
where $(\eta_j)_{j=1}^N$ are centered i.i.d.~random variables independent of $x_1, \dots, x_N$ with variance bounded by $\sigma^2>0$. Notice that the (random) function $  \widetilde{f}  $ belongs to $L^2(\SS, \pp_N)$.

\begin{theorem}\label{teo:cohen3eps}  Let $(\SS,\sfd,\pp)$ satisfy property \ref{it:H}. Let $M>0$, let $f: \SS \to [-M,M]$ be a measurable function and let $  \widetilde{f}^\star   $ be defined as
\[   \widetilde{f}^\star    := -M \vee S_N^{\lambda,n}(  \widetilde{f}  ) \wedge M,\]
where $S_N^{\lambda,n}$ is as in Definition \ref{def:someop} with $\CE^{\eps}_{2, \pp_N}$ in place of $\pCE_{2, \pp_N}$, and $  \widetilde{f}  $ is as in \eqref{eq:feta2} based on the noisy data $(\widetilde f(x_j))_{j=1}^N$ whose centered i.i.d.~errors $(\eta_j)_{j=1}^N$ are independent of $x_1, \dots, x_N$ and have variance bounded by $\sigma^2>0$. Let $r>0$ be given and assume that $n,N \in \N$ satisfy \eqref{eq:thecond}. Then
\begin{align*}
\mathbb{E} \left [ \|f-  \widetilde{f}^\star   \|_{L^2(\SS, \pp)}^2\right ] &\le 2 e(f,n) \left ( 1 + \frac{ c_{1/2}}{\log(N)(1+r)(1/2+\lambda   \mu_{\min,n}^\eps  )^2} \right ) + 8\lambda \CE_{2, \pp}(P_n f) + \\
& \quad + 4\frac{\sigma^2}{(1/2+\lambda   \mu_{\min,n}^\eps  )^2}\frac{n}{N}+ 4  M^2 (N^{-r}+N\eps),
\end{align*}
where $e(f,n)$ and $  \mu_{\min,n}^\eps  $ are as in Theorem \ref{teo:cohen2eps}.
\end{theorem}

\section*{Acknowledgments}
Funded by the European Union. Views and opinions expressed are however those of the author(s) only and do not necessarily reflect those of the European Union or the European Research Council Executive Agency. Neither the European Union nor the granting authority can be held responsible for them. This project has received funding from the European Research Council (ERC) under the European Union’s Horizon Europe research and innovation programme (grant agreement No. 101198055, project acronym NEITALG).



\begin{thebibliography}{99}

\bibitem{Amb90}
L.~Ambrosio.
\newblock {Metric space valued functions of bounded variation}.
\newblock {\em Ann. Scuola Norm. Sup. Pisa Cl. Sci. (4)}, 17(3):439--478,
  1990.

\bibitem{AmbColDma}
L.~Ambrosio, M.~Colombo, and S.~Di~Marino.
\newblock {Sobolev spaces in metric measure spaces: reflexivity and lower
  semicontinuity of slope}.
\newblock {\em Adv. Stud. Pure Math.}, 67:1--58, 2015.

\bibitem{AmbGigSav08}
L.~Ambrosio, N.~Gigli, and G.~Savar\'e.
\newblock {\em {Gradient flows in metric spaces and in the space of probability
  measures}}.
\newblock Lectures in Mathematics ETH Z\"urich. Birkh\"auser Verlag, Basel,
  second edition, 2008.

\bibitem{AmbGigSav13}
L.~Ambrosio, N.~Gigli, and G.~Savar\'e.
\newblock {Density of {L}ipschitz functions and equivalence of weak gradients
  in metric measure spaces}.
\newblock {\em Rev. Mat. Iberoam.}, 29(3):969--996, 2013.

\bibitem{AmbGigSav14}
L.~Ambrosio, N.~Gigli, and G.~Savar\'e.
\newblock {Calculus and heat flow in metric measure spaces and applications to
  spaces with {R}icci bounds from below}.
\newblock {\em Invent. Math.}, 195(2):289--391, 2014.

\bibitem{AmbIkoLucPas24}
L.~Ambrosio, T.~Ikonen, D.~Lu\v{c}i\'c, and E.~Pasqualetto.
\newblock {Metric {S}obolev spaces {I}: {E}quivalence of definitions}.
\newblock {\em Milan J. Math.}, 92(2):255--347, 2024.

\bibitem{AmbIkoLucPas25}
L.~Ambrosio, T.~Ikonen, D.~Lu\v{c}i\'c, and E.~Pasqualetto.
\newblock {Correction to: Metric Sobolev Spaces I: Equivalence of Definitions}.
\newblock {\em Milan J. Math.}, 93(2):551--555, 2025.

\bibitem{BjoBjo11}
A.~Bj\"orn and J.~Bj\"orn.
\newblock {\em {Nonlinear potential theory on metric spaces}}, volume~17 of
  {\em EMS Tracts in Mathematics}.
\newblock European Mathematical Society (EMS), Z\"urich, 2011.

\bibitem{Che99}
J.~Cheeger.
\newblock {Differentiability of {L}ipschitz functions on metric measure
  spaces}.
\newblock {\em Geom. Funct. Anal.}, 9(3):428--517, 1999.

\bibitem{ChiPeySchVia18}
L.~Chizat, G.~Peyr\'e, B.~Schmitzer, and F.~c.-X. Vialard.
\newblock {An interpolating distance between optimal transport and
  {F}isher-{R}ao metrics}.
\newblock {\em Found. Comput. Math.}, 18(1):1--44, 2018.

\bibitem{CohDavLev13}
A.~Cohen, M.~A. Davenport, and D.~Leviatan.
\newblock {On the stability and accuracy of least squares approximations}.
\newblock {\em Found. Comput. Math.}, 13(5):819--834, 2013.

\bibitem{CohDavLev19}
A.~Cohen, M.~A. Davenport, and D.~Leviatan.
\newblock {Correction to: On the stability and accuracy of least squares
  approximations}.
\newblock {\em Found. Comput. Math.}, 19(1):239, 2019.

\bibitem{DPoSodTam25}
N.~De~Ponti, G.~E. Sodini, and L.~Tamanini.
\newblock {The infimal convolution structure of the Hellinger-Kantorovich
  distance}.
\newblock 2025.
\newblock arXiv 2503.12939.

\bibitem{Dsc25}
L.~Dello~Schiavo.
\newblock {Massive Particle Systems, Wasserstein Brownian Motions, and the
  Dean-Kawasaki Equation}.
\newblock 2025.
\newblock arXiv 2411.14936.

\bibitem{DscSod25}
L.~Dello~Schiavo and G.~E. Sodini.
\newblock {The Hellinger-Kantorovich metric measure geometry on spaces of
  measures}.
\newblock 2025.
\newblock arXiv 2503.07802.

\bibitem{DMa14}
S.~Di~Marino.
\newblock {Sobolev and BV spaces on metric measure spaces via derivations and
  integration by parts}, 2014.
\newblock arXiv 1409.5620.

\bibitem{DMaGigPasSou21}
S.~Di~Marino, N.~Gigli, E.~Pasqualetto, and E.~Soultanis.
\newblock {Infinitesimal {H}ilbertianity of locally {CAT}$(\kappa)$-spaces}.
\newblock {\em J. Geom. Anal.}, 31(8):7621--7685, 2021.

\bibitem{DMaLucPas20}
S.~Di~Marino, D.~Lu\v{c}i\'c, and E.~Pasqualetto.
\newblock {A short proof of the infinitesimal {H}ilbertianity of the weighted
  {E}uclidean space}.
\newblock {\em C. R. Math. Acad. Sci. Paris}, 358(7):817--825, 2020.

\bibitem{ForHeiSod25}
M.~Fornasier, P.~Heid, and G.~E. Sodini.
\newblock {Approximation theory, computing, and deep learning on the
  {W}asserstein space}.
\newblock {\em Math. Models Methods Appl. Sci.}, 35(4):825--903, 2025.

\bibitem{ForSavSod23}
M.~Fornasier, G.~Savar\'e, and G.~E. Sodini.
\newblock {Density of subalgebras of {L}ipschitz functions in metric {S}obolev
  spaces and applications to {W}asserstein {S}obolev spaces}.
\newblock {\em J. Funct. Anal.}, 285(11):Paper No. 110153, 76, 2023.

\bibitem{Gig15}
N.~Gigli.
\newblock {On the differential structure of metric measure spaces and
  applications}.
\newblock {\em Mem. Amer. Math. Soc.}, 236(1113):vi+91, 2015.

\bibitem{GigPas20}
N.~Gigli and E.~Pasqualetto.
\newblock {\em {Lectures on nonsmooth differential geometry}}, volume~2 of {\em
  SISSA Springer Series}.
\newblock Springer, Cham, 2020.

\bibitem{HeiJuhSha15}
J.~Heinonen, P.~Koskela, N.~Shanmugalingam, and J.~T. Tyson.
\newblock {\em {Sobolev spaces on metric measure spaces}}, volume~27 of {\em
  New Mathematical Monographs}.
\newblock Cambridge University Press, Cambridge, 2015.
\newblock An approach based on upper gradients.

\bibitem{KonMonVor16}
S.~Kondratyev, L.~Monsaingeon, and D.~Vorotnikov.
\newblock {A new optimal transport distance on the space of finite {R}adon
  measures}.
\newblock {\em Adv. Differential Equations}, 21(11-12):1117--1164, 2016.

\bibitem{KosMMa98}
P.~Koskela and P.~MacManus.
\newblock {Quasiconformal mappings and {S}obolev spaces}.
\newblock {\em Studia Math.}, 131(1):1--17, 1998.

\bibitem{LieMieSav18}
M.~Liero, A.~Mielke, and G.~Savar\'e.
\newblock {Optimal entropy-transport problems and a new
  {H}ellinger-{K}antorovich distance between positive measures}.
\newblock {\em Invent. Math.}, 211(3):969--1117, 2018.

\bibitem{LuiSav21}
G.~Luise and G.~Savar\'e.
\newblock {Contraction and regularizing properties of heat flows in metric
  measure spaces}.
\newblock {\em Discrete Contin. Dyn. Syst. Ser. S}, 14(1):273--297, 2021.

\bibitem{LucPas20}
D.~Lu\v{c}i\'c and E.~Pasqualetto.
\newblock {Infinitesimal {H}ilbertianity of weighted {R}iemannian manifolds}.
\newblock {\em Canad. Math. Bull.}, 63(1):118--140, 2020.

\bibitem{PasSod25}
E.~Pasqualetto and G.~E. Sodini.
\newblock {Functions of bounded variation and Lipschitz algebras in metric
  measure spaces}.
\newblock {\em ESAIM:COCV}, 32, 2026.

\bibitem{Sav22}
G.~Savar\'e.
\newblock {Sobolev spaces in extended metric-measure spaces}.
\newblock In {\em {New trends on analysis and geometry in metric spaces}},
  volume 2296 of {\em Lecture Notes in Math.}, pages 117--276. Springer, Cham,
  2022.

\bibitem{SavSod24}
G.~Savaré and G.~E. Sodini.
\newblock {A relaxation viewpoint to unbalanced optimal transport: duality,
  optimality and Monge formulation}.
\newblock {\em J. Math. Pures Appl.}, 188:114--178, 2024.

\bibitem{Sha00}
N.~Shanmugalingam.
\newblock {Newtonian spaces: an extension of {S}obolev spaces to metric measure
  spaces}.
\newblock {\em Rev. Mat. Iberoamericana}, 16(2):243--279, 2000.

\bibitem{Sod23}
G.~E. Sodini.
\newblock {The general class of Wasserstein Sobolev spaces: density of cylinder
  functions, reflexivity, uniform convexity and Clarkson's inequalities}.
\newblock {\em Calc. Var. Partial Differential Equations}, 62(7):Paper No. 212,
  41, 2023.

\bibitem{Tro12}
J.~A. Tropp.
\newblock {User-friendly tail bounds for sums of random matrices}.
\newblock {\em Found. Comput. Math.}, 12(4):389--434, 2012. 


\end{thebibliography}
 \end{document}